\documentclass[preprint,12pt]{elsarticle}

\usepackage{graphics}
\usepackage{cancel}
\usepackage{tikz-cd} 
\usetikzlibrary{matrix,arrows}
\usepackage{amsmath}
\usepackage{amsthm}
\usepackage{amssymb}
\usepackage{enumitem}
\usepackage{amsrefs}

\newcommand{\mathbfL}{{\mathchoice{\mbox{\bf\L}}{\mbox{\bf\L}}{\mbox{\bf\scriptsize\L}}{\mbox{\bf\tiny\L}}}}

\newcommand{\lan}{\langle}
\newcommand{\ran}{\rangle}

\newcommand{\mt}{\land}
\newcommand{\jn}{\lor}

\newcommand{\m}[1]{{\bf {#1}}}
\newcommand{\prp}[1]{{\sf #1}}
\newcommand{\f}{\ensuremath{\varphi}}
\newcommand{\ps}{\ensuremath{\psi}}
\newcommand{\x}{\ensuremath{\chi}}

\DeclareMathOperator{\Con}{\mathrm{Con}}

\newcommand{\cg}[1]{{\rm Cg}_{_{#1}}}

\newcommand{\cgq}[1]{{\rm Cg}_{_{#1}}^{\Q}}

\newcommand{\cls}[1]{\mathcal{#1}}

\newcommand{\De}{\Delta}
\newcommand{\Ga}{\Gamma}

\newcommand{\The}{\Theta}

\newcommand{\N}{\mathbb{N}}
\newcommand{\V}{\mathcal{V}}
\newcommand{\Vfsi}{\mathcal{V}_{_\text{FSI}}}

\newcommand{\Vsi}{\mathcal{V}_{_\text{SI}}}

\newcommand{\K}{\mathcal{K}}
\newcommand{\Q}{\mathcal{Q}}
\newcommand{\Qfsi}{\mathcal{Q}_{_\text{FSI}}}
\newcommand{\Qrfsi}{\mathcal{Q}_{_\text{RFSI}}}
\newcommand{\Qsi}{\mathcal{Q}_{_\text{SI}}}
\newcommand{\Qrsi}{\mathcal{Q}_{_\text{RSI}}}

\DeclareMathOperator{\Conq}{\mathrm{Con}_{\Q}}

\theoremstyle{definition}
\newtheorem{Theorem}{Theorem}
\newtheorem{Proposition}[Theorem]{Proposition}
\newtheorem{Lemma}[Theorem]{Lemma}

\newtheorem{Remark}[Theorem]{Remark}
\newtheorem{Corollary}[Theorem]{Corollary}

\numberwithin{Theorem}{section}

\setlist[itemize]{leftmargin=*, itemsep=0.1cm, topsep=0cm}
\setlist[enumerate]{leftmargin=1cm, itemsep=0cm, topsep=0.15cm}

\journal{Elsevier} 

\begin{document}

\begin{frontmatter}


\title{Transfer theorems for finitely subdirectly irreducible algebras}

\author{Wesley Fussner}
\address{Mathematical Institute, University of Bern, Switzerland}
\ead{wesley.fussner@unibe.ch}

\author{George Metcalfe}
\address{Mathematical Institute, University of Bern, Switzerland}
\ead{george.metcalfe@unibe.ch}

\begin{abstract}
We show that under certain conditions, well-studied algebraic properties transfer from the class $\Qrfsi$ of the relatively finitely subdirectly irreducible members of a quasivariety $\Q$ to the whole quasivariety, and, in certain cases, back again. First, we prove that if $\Q$ is relatively congruence-distributive, then it has the $\Q$-congruence extension property ($\Q$-\prp{CEP}) if and only if $\Qrfsi$ has this property. We then prove that if $\Q$ has the $\Q$-\prp{CEP} and $\Qrfsi$ is closed under subalgebras, then $\Q$ has a one-sided amalgamation property (equivalently, for $\Q$, the amalgamation property) if and only if $\Qrfsi$ has this property. We also establish similar results for the transferable injections property and strong amalgamation property. For each property considered, we specialize our results to the case where $\Q$ is a variety --- so that $\Qrfsi$ is the class of finitely subdirectly irreducible members of $\Q$ and the $\Q$-\prp{CEP} is the usual congruence extension property --- and prove that when $\Q$ is finitely generated and congruence-distributive, and $\Qrfsi$ is closed under subalgebras, possession of the property is decidable. Finally, as a case study, we provide a complete description of the subvarieties of a notable variety of BL-algebras that have the amalgamation property.
\end{abstract}


\end{frontmatter}


\section{Introduction} \label{s:introduction}

This paper studies a cluster of interrelated properties of general interest in algebra: the congruence extension property, (strong) amalgamation property,  transferable injections property, and surjective epimorphisms property. The latter have each been investigated in a range of algebraic contexts, with studies spanning groups, rings, lattices, Lie algebras, and numerous other algebraic structures (see~\cite{KMPT83} for an extensive survey). In this paper, we adopt a general vantage point and develop a widely applicable toolkit for studying these properties. Our main contribution consists of transfer theorems that characterize each property for a class of algebraic structures satisfying certain hypotheses in terms of a smaller and more easily studied subclass, namely, the (relatively) finitely subdirectly irreducible members of the class. These theorems generalize other transfer results in the literature (found in, e.g.,~\cites{GL71,Dav77,Kis85b,CP99,MMT14}), while simultaneously providing a more uniform treatment. 

We freely make use of basic notions of universal algebra throughout our discussion, but first recall several key definitions, referring to~\cites{BS81,Gor98} for more detailed treatments. Let $\Q$ be a \emph{quasivariety}: a class of similar algebras (i.e., a class of algebras of the same signature) defined by quasiequations, or, equivalently, closed under isomorphisms, subalgebras, direct products, and ultraproducts. A congruence $\The$ of an algebra $\m{A}\in\Q$  is called a \emph{$\Q$-congruence} if $\m{A}/\The\in\Q$. When ordered by inclusion, the set of $\Q$-congruences of $\m{A}$ forms an algebraic lattice $\Conq\m{A}$. Arbitrary meets coincide in $\Conq\m{A}$ with those taken in the algebraic lattice $\Con\m{A}$ of all congruences of $\m{A}$, as do joins of chains, but this may not be the case for arbitrary joins. If $\Q$ is a {\em variety} --- that is, a class of similar algebras defined by equations, or, equivalently, closed under homomorphic images, subalgebras, and direct products --- then every congruence of $\m{A}$ is a $\Q$-congruence and $\Conq\m{A}$ and $\Con\m{A}$ coincide. 

An algebra $\m{A}$ is said to be \emph{(finitely) subdirectly irreducible} if whenever $\m{A}$ is isomorphic to a subdirect product of a (non-empty finite) set of algebras, it is isomorphic to one of these algebras. Equivalently, $\m{A}$ is finitely subdirectly irreducible if the least congruence $\Delta_A := \{\lan a,a\ran\mid a\in A\}$ is meet-irreducible in $\Con\m{A}$, and subdirectly irreducible if $\Delta_A$ is completely meet-irreducible in $\Con\m{A}$.\footnote{An element $a$ of a lattice $\m{L}$ is {\em meet-irreducible} if $a=b\mt c$ implies $a=b$ or $a=c$, and this is true of any greatest element $\top$ of $\m{L}$; however, $a$ is {\em completely meet-irreducible} if $a=\bigwedge B$ implies $a\in B$ for any $B\subseteq L$, which is not the case for $\top=\bigwedge\emptyset$.} When ${\m A}$ belongs to $\Q$, it is convenient to relativize these notions. In this case, ${\m A}$ is said to be \emph{(finitely) $\Q$-subdirectly irreducible} if whenever $\m{A}$ is isomorphic to a subdirect product of a (non-empty finite) set of algebras in $\Q$, it is isomorphic to one of them. Equivalently, $\m{A}$ is finitely $\Q$-subdirectly irreducible if and only if $\Delta_A$ is meet-irreducible in $\Conq\m{A}$, and $\Q$-subdirectly irreducible if and only if $\Delta_A$ is completely meet-irreducible in $\Conq\m{A}$.

For any $\m{A}\in\Q$ and $\The\in\Conq\m{A}$, it follows from the correspondence theorem for universal algebra that the quotient algebra $\m{A}/\The\in\Q$ is finitely $\Q$-subdirectly irreducible if and only if $\The$ is meet-irreducible in $\Conq\m{A}$, and $\Q$-subdirectly irreducible if and only if it is completely meet-irreducible. Clearly, if $\Q$ is a variety, the properties of being (finitely) $\Q$-subdirectly irreducible and (finitely) subdirectly irreducible  coincide. When $\Q$ is clear from the context, we call a (finitely) $\Q$-subdirectly irreducible algebra $\m{A}$ \emph{relatively (finitely) subdirectly irreducible}.

Let $\Qfsi$, $\Qsi$, $\Qrfsi$, and $\Qrsi$ denote the classes of finitely subdirectly irreducible, subdirectly irreducible, relatively finitely subdirectly irreducible, and relatively subdirectly irreducible members of $\Q$, respectively. A sizeable number of results in the universal algebra literature state that, under certain conditions, well-studied algebraic properties transfer from $\Qrsi$ to $\Q$, at least when $\Q$ is a variety (see, e.g.,~\cites{GL71,Dav77,Kis85b,CP99,MMT14}). The aim of this paper is to determine conditions under which these properties transfer from $\Qrfsi$ to $\Q$ and, in some cases, back again. A key motivation for considering $\Qrfsi$ rather than $\Qrsi$ is that it is often easier to state or check conditions for the larger class. Notably, if $\Q$ has equationally definable relative principal congruence meets (a common property for quasivarieties corresponding to non-classical logics), then $\Qrfsi$ is a universal class~\cite{CD90}*{Theorem~2.3}. Moreover, if $\V$ is a variety such that $\Vfsi$ is a universal class, then $\Vfsi$ is a positive universal class if and only if $\Con{\m{A}}$ is a chain (totally ordered set) for each $\m{A}\in\Vfsi$.\footnote{Suppose that $\Vfsi$ is a universal class. Then it is a positive universal class if and only if for any $\m{A}\in\Vfsi$ and $\The\in\Con\m{A}$, also $\m{A}/\The\in\Vfsi$, i.e., $\The$ is meet-irreducible in $\Con\m{A}$. But a lattice is a chain if and only if all its elements are meet-irreducible, so  $\Vfsi$ is a positive universal class if and only if $\Con\m{A}$ is a chain for each $\m{A}\in\Vfsi$.}

An algebra $\m{B}\in\Q$ is said to be \emph{$\Q$-congruence-distributive} if $\Conq\m{B}$ is a distributive lattice. If every member of $\Q$ is $\Q$-congruence-distributive, then $\Q$ is said to be \emph{relatively congruence-distributive} and in this case, as shown in~\cite{Dz89}*{Theorem 2.3}, $\Qrfsi=\Qfsi$. Clearly, a variety is relatively congruence-distributive if and only if it is congruence-distributive in the usual sense.

An algebra $\m{B}\in\Q$ is said to have the {\em $\Q$-congruence extension property} (for short, $\Q$-\prp{CEP}) if for any subalgebra $\m{A}$ of $\m{B}$ and $\The\in\Conq\m{A}$, there exists a $\Phi\in\Conq\m{B}$ such that $\Phi\cap A^2=\The$. A class $\K$ of algebras in $\Q$ is said to have the $\Q$-\prp{CEP} if every member of $\K$ has the $\Q$-\prp{CEP}. In Section~\ref{s:cep}, we prove that if $\Q$ is relatively congruence-distributive, it has the $\Q$-\prp{CEP} if and only if $\Qrfsi=\Qfsi$ has the $\Q$-\prp{CEP} (Theorem~\ref{t:CEPmain}). When $\Q$ is a variety, the $\Q$-\prp{CEP} is the usual congruence extension property (for short, \prp{CEP}) and if it is congruence-distributive, $\Q$ has the \prp{CEP} if and only if $\Qfsi$ has the \prp{CEP} (Corollary~\ref{c:CEPvar}). This result yields also the known fact that a congruence-distributive variety $\V$ such that $\Vsi$ is elementary has the \prp{CEP} if and only if $\Vsi$ has the \prp{CEP}~\cite{Dav77}*{Theorem 3.3}. Note, however, that the requirement that $\Vsi$ is elementary may fail or be difficult to establish, and this result may therefore be significantly harder to apply than Theorem~\ref{t:CEPmain}.

When $\Qfsi$ is closed under subalgebras, Theorem~\ref{t:CEPmain} can be reformulated in terms of commutative diagrams. Let $\K$ be any class of similar algebras. A \emph{span} in $\K$ is a 5-tuple $\lan\m{A},\m{B},\m{C},\f_B,\f_C \ran$ consisting of $\m{A},\m{B},\m{C}\in \K$ and homomorphisms $\f_B\colon\m{A}\to\m{B}$, $\f_C\colon\m{A}\to\m{C}$. We call this span  {\em injective} if $\f_B$ is an embedding, {\em doubly injective} if both $\f_B$ and $\f_C$ are embeddings, and {\em injective-surjective} if $\f_B$ is an embedding and $\f_C$ is surjective. The class $\K$ has the {\em extension property}  (for short, \prp{EP}) if for any injective-surjective span $\lan\m{A},\m{B},\m{C},\f_B,\f_C \ran$ in $\K$, there exist a $\m{D}\in\K$, a homomorphism $\ps_B\colon\m{B}\to\m{D}$, and an embedding $\ps_C\colon\m{C}\to\m{D}$ such that $\ps_B \f_B=\ps_C\f_C$, that is, the diagram in Figure~\ref{f:properties}(i) is commutative. We prove that $\Q$ has the $\Q$-\prp{CEP} if and only if it has the \prp{EP} (Corollary~\ref{c:rcep iff ep}, generalizing~\cite{Bac72}*{Lemma~1.2}) and that if $\Q$ is relatively congruence-distributive and $\Qrfsi=\Qfsi$ is closed under subalgebras, then the \prp{EP} and $\Q$-\prp{CEP} for $\Q$ and $\Qrfsi=\Qfsi$ all coincide (Theorem~\ref{t:sepcep}).

\begin{figure}[t]
\centering
\begin{scriptsize}
$\begin{array}{ccccccc}
\begin{tikzcd}
& \m C \ar[dashed, hook, rd, "\ps_C"] &       \\
\m{A} \ar[->>, ru, "\f_C"] \ar[hook, rd,"\f_B"'] &      & \m D  \\
 & \m{B} \ar[dashed, ru, "\ps_B"'] &     
\end{tikzcd}
&&
\begin{tikzcd}
& \m C \ar[dashed, hook, rd, "\ps_C"] &       \\
\m{A} \ar[hook, ru, "\f_C"] \ar[hook, rd,"\f_B"'] &      & \m D  \\
 & \m{B} \ar[dashed, hook, ru, "\ps_B"'] &     
\end{tikzcd}
& &
\begin{tikzcd}
& \m C \ar[dashed, hook, rd, "\ps_C"] &       \\
\m{A} \ar[hook, ru, "\f_C"] \ar[hook, rd,"\f_B"'] &      & \m D  \\
 & \m{B} \ar[dashed,  ru, "\ps_B"'] &     
\end{tikzcd}
& &
\begin{tikzcd}
& \m C \ar[dashed, hook, rd, "\ps_C"] &       \\
\m{A} \ar[ ru, "\f_C"] \ar[hook, rd,"\f_B"'] &      & \m D  \\
 & \m{B} \ar[dashed,  ru, "\ps_B"'] &     
\end{tikzcd}\\[.5in]
\text{(i)} \enspace \prp{EP}  \quad& & \text{(ii)} \enspace\prp{AP} \quad & & \text{(iii)} \enspace\prp{1AP} \quad && \text{(iv)} \enspace\prp{TIP} \quad
\end{array}$
\end{scriptsize}
\caption{Commutative diagrams for algebraic properties}
\label{f:properties}
\end{figure}
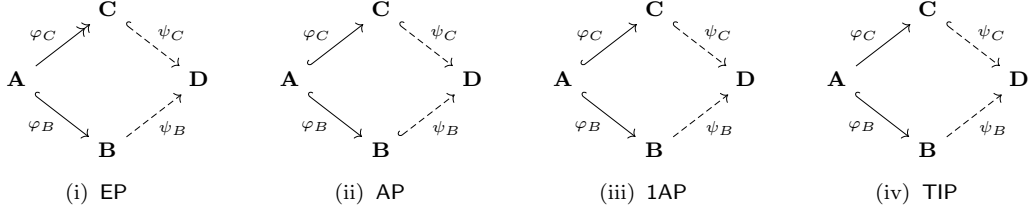

Now let $\K$ and $\K'$ be two classes of algebras of the same signature.  An {\em amalgam in }$\K'$ of a doubly injective span $\lan\m{A},\m{B},\m{C},\f_B,\f_C \ran$ in $\K$ is a triple $\lan \m{D},\ps_B,\ps_C\ran$ where $\m{D}\in\K'$ and $\ps_B,\ps_C$ are embeddings of $\m{B}$ and $\m{C}$ into $\m{D}$, respectively, such that $\ps_B\f_B = \ps_C\f_C$ (see Figure~\ref{f:properties}(ii)). The class $\K$ has the {\em amalgamation property}  (for short, \prp{AP})  if every doubly injective span in $\K$ has an amalgam in $\K$. We also say that $\K$  has the {\em one-sided amalgamation property} (for short, \prp{1AP})  if for any doubly injective span $\lan\m{A},\m{B},\m{C},\f_B,\f_C \ran$ in $\K$, there exist a $\m{D}\in\K$,  a homomorphism $\ps_B\colon\m{B}\to\m{D}$, and an embedding $\ps_C\colon\m{C}\to\m{D}$ such that $\ps_B \f_B=\ps_C\f_C$  (see Figure~\ref{f:properties}(iii)). It is easy to see using~\cite{GL71}*{Lemma~2} that a variety $\V$ has the~\prp{1AP} if and only if it has the~\prp{AP}, but this is not always the case for other classes, in particular, $\Vfsi$. In Section~\ref{s:ap}, we prove that when $\Q$ has the $\Q$-\prp{CEP} and $\Qrfsi$ is closed under subalgebras, $\Q$ has the~\prp{1AP} (equivalently, the~\prp{AP}) if and only if $\Qrfsi$ has the~\prp{1AP} (Theorem~\ref{t:APmain}). 

In Section~\ref{s:further}, we consider consequences of our results for three further properties. First, a class $\K$ of similar algebras is said to have the \emph{transferable injections property} (for short, \prp{TIP}) if for any injective span $\lan\m{A},\m{B},\m{C},\f_B,\f_C \ran$ in $\K$, there exist a $\m{D}\in \K$, a homomorphism $\ps_B\colon\m{B}\to\m{D}$, and an embedding $\ps_C\colon\m{C}\to\m{D}$ such that $\ps_B \f_B=\ps_C\f_C$ (see Figure~\ref{f:properties}(iv)). It is well-known that a variety has the \prp{TIP} if and only if it has the \prp{CEP} and \prp{AP} (\cite{Bac72}*{Lemma~1.7}); more generally, as we show here, a class of similar algebras that is closed under subalgebras has the \prp{TIP} if and only if it has the \prp{EP} and \prp{1AP}. It then follows from our previous results that when $\Q$ is relatively congruence-distributive and $\Qrfsi=\Qfsi$ is closed under subalgebras, $\Q$ has the \prp{TIP} if and only if $\Qrfsi=\Qfsi$ has the \prp{TIP} (Theorem~\ref{t:mainTIP}). 

Next, let $\K$ be a class of similar algebras and $\m{A},\m{B}\in\K$. A homomorphism $\varphi\colon\m{A}\to\m{B}$ is an \emph{epimorphism in $\K$} if for all $\m{C}\in\K$ and all homomorphisms $\ps_1,\ps_2\colon\m{B}\to\m{C}$, if $\ps_1\varphi=\ps_2\varphi$, then $\ps_1=\ps_2$. Surjective homomorphisms are always epimorphisms, but the converse does not hold in general. If all epimorphisms in $\K$ are surjections, $\K$ is said to have \emph{surjective epimorphisms} (for short, \prp{SE}). In~\cite{Cam18}*{Theorem~22}, it was proved that an arithmetical (i.e., congruence-distributive and congruence-permutable) variety $\V$ such that $\Vfsi$ is a universal class has \prp{SE} if and only if $\Vfsi$ has \prp{SE}. 

For a class of algebras $\K'$ of the same signature as $\K$, an amalgam $\lan \m{D},\ps_B,\ps_C\ran$ in $\K'$ of a doubly injective span $\lan\m{A},\m{B},\m{C},\f_B,\f_C \ran$ in $\K$ is called {\em strong} if  $\ps_B\f_B[A] = \ps_B[B]\cap\ps_C[C]$. The class $\K$ is said to have the \emph{strong amalgamation property} (for short, \prp{SAP}) if every doubly injective span in $\K$ has a strong amalgam in $\K$. Using the fact that a quasivariety has the  \prp{SAP} if and only if it has \prp{SE} and the \prp{AP}~\cite{Isb1966}, it follows from our previous results and~\cite{Cam18}*{Theorem~22} that an arithmetical variety $\V$ with the \prp{CEP} such that $\Vfsi$ is a universal class has the \prp{SAP} if and only if $\Vfsi$ has \prp{SE} and the \prp{1AP} (Corollary~\ref{c:mainsa}). We also show that such a variety has the \prp{SAP} if every doubly injective span in $\Vfsi$ has a strong amalgam in $\V$ (Theorem~\ref{t:mainsa}).

In Section~\ref{s:decidability}, we conclude that possession of all the properties mentioned above is decidable for certain finitely generated varieties. More precisely, we obtain effective algorithms to decide if a congruence-distributive variety $\V$ that is finitely generated by a given finite set of finite algebras, such that $\Vfsi$ is closed under subalgebras, has the \prp{CEP},  \prp{AP}, or \prp{TIP} (Theorem~\ref{t:decidability}). In the case where $\V$ is arithmetical, we obtain also effective algorithms to decide if $\V$  has \prp{SE} or the \prp{SAP}. Finally, in Section~\ref{s:BLalgebras}, we provide a complete description of the subvarieties of a notable variety of BL-algebras (those generated by a class of ``one-component'' totally ordered BL-algebras) that have the \prp{AP}.


\section{The Congruence Extension Property} \label{s:cep}

We first recall some basic facts about extending congruences, denoting the $\Q$-congruence of an algebra $\m{A}\in\Q$ generated by a set $R\subseteq A^2$ by $\cgq{\m{A}}(R)$.

\begin{Lemma}\label{l:cepbasic}
Let $\Q$ be any quasivariety and let $\m{B}\in\Q$.
\begin{enumerate}[label=(\alph*)]
\item (cf.~\cite{Kis85b}*{Lemma~1.3}) 
Suppose that for any subalgebra $\m{A}$ of $\m{B}$ and completely meet-irreducible $\The\in\Conq\m{A}$, there exists a $\Psi\in\Conq\m{B}$ with $\Psi\cap A^2=\The$. Then $\m{B}$ has the $\Q$-\prp{CEP}.
\item (cf.~\cite{BP97}*{p.~392}) 
Let $\m{A}$ be a subalgebra of $\m{B}$ and $\The\in\Conq\m{A}$ such that $\Psi\cap A^2=\The$ for some $\Psi\in\Conq\m{B}$. Then $\cgq{\m{B}}(\The)\cap A^2 = \The$.
\end{enumerate}
\end{Lemma}
\begin{proof}
(a) Consider any subalgebra $\m{A}$ of $\m{B}$ and $\The\in\Conq\m{A}$. Since $\Conq\m{A}$ is an algebraic lattice, there exists a set $\{\The_i\}_{i\in I}$ of $\Q$-congruences of $\m{A}$ such that $\The=\bigcap_{i\in I}\The_i$ and $\The_i$ is  completely meet-irreducible for each $i\in I$ (see, e.g.,~\cite{Gor98}*{Lemma~1.3.2}). By assumption, there exists for each $i\in I$, a $\Psi_i\in\Conq\m{B}$ with $\Psi_i\cap A^2=\The_i$. Hence $\Psi:=\bigcap_{i\in I}\Psi_i\in\Conq\m{B}$ and $\Psi\cap A^2=\bigcap_{i\in I}(\Psi_i\cap A^2)=\bigcap_{i\in I}\The_i=\The$.  

(b) Since $\The\subseteq\Psi$ and $\Psi\in\Conq\m{B}$, also $\cgq{\m{B}}(\The)\subseteq\Psi$. Hence, using the assumption, $\The\subseteq\cgq{\m{B}}(\The)\cap A^2 \subseteq\Psi\cap A^2=\The$.
\end{proof}

We also make use of the following consequence of the correspondence theorem of universal algebra.

\begin{Lemma}[cf.~\cite{BP97}*{Lemma~2}]\label{l:corresp}
Let $\Q$ be any quasivariety and let $\m{A},\m{B}\in\Q$. For any surjective homomorphism $\f\colon\m{A}\to\m{B}$ and $R\subseteq A\times A$,
\[
\f^{-1}[\cgq{\m{B}}(\f[R])]=\cgq{\m{A}}(R)\jn\ker(\f),
\]
where the join on the right hand side is taken in $\Conq\m{A}$ and $\f[R]$ abbreviates $\{\lan\f(x),\f(y)\ran\mid\lan x,y\ran\in R\}$.
\end{Lemma}
\begin{proof}
Let $\The:=\f^{-1}[\cgq{\m{B}}(\f[R])]$. Since $R\subseteq\The$ and $\The\in\Conq(\m{A})$, also  $\cgq{\m{A}}(R)\subseteq\The$. Moreover, $\ker(\f)\subseteq\The$, so $\cgq{\m{A}}(R)\jn\ker(\f)\subseteq\The$. For the converse inclusion, since $\cgq{\m{A}}(R)\jn\ker(\f)\in [\ker(\f),A^2]$, it follows using the correspondence theorem that
\[
\f^{-1}[\cgq{\m{B}}(\f[R])]\subseteq\f^{-1}[\cgq{\m{B}}(\f[\cgq{\m{A}}(R)\jn\ker(\f)])]=\cgq{\m{A}}(R)\jn\ker(\f).  \qedhere
\]
\end{proof}

We now establish the first main result of this section, recalling from the introduction that if $\Q$ is a relatively congruence-distributive quasivariety, then $\Qrfsi=\Qfsi$~\cite{Dz89}*{Theorem 2.3}.

\begin{Theorem}\label{t:CEPmain}
Let $\Q$ be any relatively congruence-distributive quasivariety.  Then $\Q$ has the $\Q$-congruence extension property if and only if $\Qrfsi=\Qfsi$ has the $\Q$-congruence extension property. 
\end{Theorem}
\begin{proof}
Suppose for the non-trivial direction that $\Qrfsi$ has the $\Q$-\prp{CEP}. Let  $\m{A}$ be a subalgebra of  some $\m{B}\in\Q$ and let $\The\in\Conq\m{A}$. We assume towards a contradiction that $\cgq{\m{B}}(\The)\cap A^2 \neq \The$; that is, there exists an ordered pair $\lan a,b\ran\in\cgq{\m{B}}(\The)\cap A^2$ satisfying $\lan a,b\ran\not\in\The$. Define
\[
T:=\{\Psi\in\Conq\m{B}\mid\lan a,b\ran\not\in(\Psi\cap A^2)\jn\The\}.
\]
Then $\Delta_B\in T$, so $T\neq\emptyset$. Moreover, every chain in $\lan T,\subseteq\ran$ has an upper bound (its union) in $T$, so, by Zorn's Lemma, $\lan T,\subseteq\ran$ has a maximal element~$\Psi^*$.

\medskip\noindent
{\bf Claim.} $\Psi^*$ is meet-irreducible in $\Conq\m{B}$ and hence $\m{B}/\Psi^*\in\Qrfsi$.

\medskip\noindent
{\em Proof of claim.} Suppose that $\Psi^*=\Psi_1\cap\Psi_2$ for some $\Psi_1,\Psi_2\in\Conq\m{B}$. Then, since $\Conq\m{A}$ is distributive by assumption,
\[
((\Psi_1\cap A^2)\jn\The)\cap((\Psi_2\cap A^2)\jn\The) = (\Psi_1\cap \Psi_2\cap A^2)\jn\The = (\Psi^*\cap A^2)\jn\The.
\]
But $\lan a,b\ran\not\in(\Psi^*\cap A^2)\jn\The$, so $\lan a,b\ran\not\in(\Psi_1\cap A^2)\jn\The$ or $\lan a,b\ran\not\in(\Psi_2\cap A^2)\jn\The$. Hence $\Psi_1\in T$ or $\Psi_2\in T$ and, by the maximality of $\Psi^*$ in $\lan T,\subseteq\ran$, either $\Psi^*=\Psi_1$ or $\Psi^*=\Psi_2$. So  $\Psi^*$ is meet-irreducible.
\qed

\medskip\noindent
Observe next that $\m{A}/(\Psi^*\cap A^2)$ embeds into $\m{B}/\Psi^*$ and can be identified with a subalgebra of  $\m{B}/\Psi^*$ with universe $A/\Psi^*$. We consider the congruence $((\Psi^*\cap A^2)\jn\The) / (\Psi^*\cap A^2)$ of $\m{A}/(\Psi^*\cap A^2)$. By the second isomorphism theorem of universal algebra,
\[
(\m{A}/(\Psi^*\cap A^2))/((\Psi^*\cap A^2)\jn\The) / (\Psi^*\cap A^2))\cong\m{A}/((\Psi^*\cap A^2)\jn\The)\in\Q,
\]
so, in particular, $(\Psi^*\cap A^2)\jn\The) / (\Psi^*\cap A^2)\in\Conq \m{A}/(\Psi^*\cap A^2)$. Since $\lan a,b\ran\not\in(\Psi^*\cap A^2)\jn\The$, also
\[
\lan [a]_{\Psi^*\cap A^2},[b]_{\Psi^*\cap A^2}\ran\not\in ((\Psi^*\cap A^2)\jn\The) / (\Psi^*\cap A^2).
\]
Recall now that $\lan a,b\ran\in\cgq{\m{B}}(\The)\cap A^2$, so $\lan a,b\ran\in\cgq{\m{B}}((\Psi^*\cap A^2)\jn\The)\jn\Psi^*$. 
Letting $R:=(\Psi^*\cap A^2)\jn\The$ and $\f$ be the canonical homomorphism from $\m{B}$ to $\m{B}/\Psi^*$ with $\ker(\f)=\Psi^*$, an application of Lemma~\ref{l:corresp} yields
\[
\f^{-1}[\cgq{\m{B}/\Psi^*}(\f[(\Psi^*\cap A^2)\jn\The])]=\cgq{\m{B}}((\Psi^*\cap A^2)\jn\The)\jn\Psi^*.
\]
Hence, identifying the congruence $((\Psi^*\cap A^2)\jn\The) / (\Psi^*\cap A^2)$ of $\m{A}/(\Psi^*\cap A^2)$ with the corresponding subset of $B/\Psi^*$,
\[
\lan [a]_{\Psi^*},[b]_{\Psi^*}\ran\in\cgq{\m{B}/\Psi^*}(((\Psi^*\cap A^2)\jn\The) / (\Psi^*\cap A^2)).
\]
But, by Lemma~\ref{l:cepbasic}, this contradicts the assumption that $\m{B}/\Psi^*\in\Qrfsi$ has the $\Q$-\prp{CEP}.
\end{proof}

\begin{Corollary}\label{c:CEPvar}
Let $\V$ be any congruence-distributive variety.  Then $\V$ has the congruence extension property if and only if $\Vfsi$ has the congruence extension property. 
\end{Corollary}

Note also that every finitely subdirectly irreducible member of a variety $\V$ embeds into an ultraproduct of subdirectly irreducible members of $\V$, by the Relativized J{\'o}nsson Lemma~\cite{CD90}*{Lemma~1.5}. Hence, if $\Vsi$ is closed under ultraproducts, each $\m{A}\in\Vfsi$ embeds into some $\m{B}\in\Vsi$, and the fact that the \prp{CEP} is preserved under subalgebras yields the following known result.

\begin{Corollary}[\cite{Dav77}*{Theorem 3.3}]\label{c:davey}
Let $\V$ be any congruence-distributive variety such that $\Vsi$ is an elementary class. Then $\V$ has the congruence extension property if and only if $\Vsi$ has the congruence extension property.\footnote{This result also follows from a more general theorem of Kiss~\cite{Kis85b}*{Theorem~2.3} for congruence-modular varieties; however, the latter does not imply, or seem to be implied by, our Theorem~\ref{t:CEPmain}.}
\end{Corollary}

We now turn our attention to the relationship between the $\Q$-\prp{CEP} and the \prp{EP}, establishing first a simple lemma and useful corollary for investigating relatively finitely subdirectly irreducible algebras.

\begin{Lemma}\label{l:key}
Let $\Q$ be any quasivariety and let $\m{A}$ be a subalgebra of some $\m{B}\in\Q$. For any meet-irreducible $\The\in\Conq\m{A}$ satisfying $\cgq{\m{B}}(\The)\cap A^2 = \The$, there exists a meet-irreducible $\Phi\in\Conq\m{B}$ such that $\Phi\cap A^2 = \The$.
\end{Lemma}
\begin{proof}
Consider any meet-irreducible $\The\in\Conq\m{A}$ satisfying $\cgq{\m{B}}(\The)\cap A^2 = \The$. By assumption, $T:=\{\Psi\in\Conq\m{B}\mid\Psi \cap A^2 =\The\}\neq\emptyset$, and, since every chain in $\lan T,\subseteq\ran$ has an upper bound (its union) in $T$, by Zorn's Lemma, $\lan T,\subseteq\ran$ has a maximal element $\Phi$. 

It remains to show that $\Phi$ is meet-irreducible in $\Conq\m{B}$, so let $\Phi=\Phi_1\cap\Phi_2$ for some $\Phi_1,\Phi\in\Conq\m{B}$. Then
\[
(\Phi_1\cap A^2)\cap(\Phi_2\cap A^2)=\Phi_1\cap\Phi_2 \cap A^2 =\Phi\cap A^2=\The
\]
and, since $\The$ is meet-irreducible in $\Conq\m{A}$, either $\Phi_1\cap A^2=\The$ or $\Phi_2\cap A^2=\The$. So $\Phi_1\in T$ or $\Phi_2\in T$. Hence, by the maximality of $\Phi$ in $\lan T,\subseteq\ran$, either $\Phi_1 =\Phi$ or  $\Phi_2 =\Phi$. So $\Phi$ is meet-irreducible in $\Conq\m{B}$.
\end{proof}

\begin{Corollary}\label{c:key}
Let $\Q$ be any quasivariety and suppose that $\m{A}\in\Qrfsi$ is a subalgebra of some $\m{B}\in\Q$. Then there exists a meet-irreducible $\Phi\in\Conq\m{B}$ such that $\Phi\cap A^2 = \De_A$, and hence there exist also a $\m{C}\in\Qrfsi$ and a surjective homomorphism $\f\colon\m{B}\to\m{C}$ such that $\ker(\f)\cap A^2 = \De_A$.
\end{Corollary}

The following result provides a general sufficient criterion for a subclass $\K$ of a quasivariety $\Q$ to have the \prp{EP}. 

\begin{Proposition}\label{p:cepsep}
Let $\K$ be a subclass of a quasivariety $\Q$ satisfying 
\begin{enumerate}[label=(\roman*)]
\item $\K$ is closed under isomorphisms; 
\item for any $\m{B}\in\Q$ and subalgebra $\m{A}\in\K$ of $\m{B}$, there exists a $\Phi\in\Conq\m{B}$ such that $\m{B}/\Phi\in \K$ and $\Phi \cap A^2 =\Delta_A$;
\item $\K$ has the $\Q$-congruence extension property.
\end{enumerate}
Then $\K$  has the extension property.
\end{Proposition}
\begin{proof}
Consider any $\m{A},\m{B},\m{C}\in\K$, embedding $\f_B\colon\m{A}\to\m{B}$, and surjective homomorphism $\f_C\colon\m{A}\to\m{C}$. Since $\K$ is closed under isomorphisms, by (i), we may assume without loss of generality that $\m{A}$ is a subalgebra of $\m{B}$ and that $\m{C}=\m{A}/\The$ for some $\The\in\Conq\m{A}$. Since $\K$ has the $\Q$-\prp{CEP}, by (iii), there exists a $\Psi\in\Conq\m{B}$ such that $\Psi\cap A^2 = \The$. Let $\m{D}:=\m{B}/\Psi\in\Q$ and let $\ps_B$ be the canonical homomorphism from $\m{B}$ to $\m{D}$ mapping each $b\in B$ to $[b]_\Psi\in B/\Psi$. Observe also that for any $a,b\in A$,
\[
[a]_\The=[b]_\The \iff \lan a,b\ran\in\The \iff \lan a,b\ran\in\Psi\iff [a]_\Psi=[b]_\Psi. 
\]
Hence we obtain an embedding $\ps_C$ of $\m{C}$ into $\m{D}$ mapping each $[a]_\The\in A/\The$ to $[a]_\Psi\in B/\Psi$ such that $\ps_B \f_B=\ps_C\f_C$. Finally, by (ii), there exist a $\m{D^*}\in\K$, a surjective homomorphism $\ps^*_B\colon\m{D}\to\m{D^*}$, and an embedding $\ps^*_C\colon\m{C}\to\m{D^*}$ such that $\ps^*_B \ps_B\f_B=\ps^*_C\f_C$. So $\K$ has the \prp{EP}.
\end{proof}

In particular, combining Proposition~\ref{p:cepsep} with Corollary~\ref{c:key}, we obtain the following result for the class of relatively finitely subdirectly irreducible members of a quasivariety.

\begin{Corollary}\label{c:cepsep}
Let $\Q$ be any quasivariety. If $\Qrfsi$ has the $\Q$-congruence extension property, then $\Qrfsi$ has the extension property.
\end{Corollary}

Next, we provide a general sufficient criterion for a subclass $\K$ of a quasivariety $\Q$ to have the $\Q$-\prp{CEP}.

\begin{Proposition}\label{p:rcep iff ep}
Let $\K$ be a subclass of a quasivariety $\Q$ satisfying 
\begin{enumerate}[label=(\roman*)]
\item $\K$ is closed under subalgebras; 
\item every relatively subdirectly irreducible member of $\Q$ belongs to $\K$;
\item $\K$ has the extension property.
\end{enumerate}
Then $\K$ has the $\Q$-congruence extension property.
\end{Proposition}
\begin{proof}
Consider any $\m{B}\in\K$. To show that $\m{B}$ has the $\Q$-\prp{CEP}, it suffices, by Lemma~\ref{l:cepbasic}(a), to prove that for any subalgebra $\m{A}$ of $\m{B}$ and completely meet-irreducible $\The\in\Conq\m{A}$, there exists a $\Psi\in\Conq\m{B}$ such that $\Psi\cap A^2 = \The$. Note first that $\m{A}\in\K$, by (i), and $\m{A}/\The\in\Qrsi\subseteq\K$, by (ii). Now let $\f_C\colon\m{A}\to\m{A}/\The$ be the canonical homomorphism mapping $a\in A$ to $[a]_\The\in A/\The$ and let $\f_B\colon\m{A}\to\m{B}$ be the inclusion map. Since $\K$ has the \prp{EP}, by (iii), there exist a $\m{D}\in\K$, a homomorphism $\ps_B\colon\m{B}\to\m{D}$, and an embedding $\ps_C\colon\m{A}/\The\to\m{D}$ such that $\ps_B \f_B=\ps_C\f_C$. Let $\Psi :=\ker(\ps_B)$. Then, by the homomorphism theorem, $\m{B}/\Psi$ is isomorphic to a subalgebra of $\m{D}\in\K$, so $\m{B}/\Psi\in\Q$ and $\Psi\in\Conq\m{B}$. Moreover, for any $a,b\in A$, using the injectivity of $\ps_C$ for the third equivalence,
\begin{align*}
\lan a,b\ran\in \Psi
& \:\iff\:  \ps_B \f_B(a)=\ps_B \f_B(b)\\
& \:\iff\:  \ps_C \f_C(a)=\ps_C \f_C(b)\\
& \:\iff\: \f_C(a)= \f_C(b)\\
& \:\iff\: \lan a,b\ran\in\ker(\f_C)=\The.
\end{align*}
That is, $\Psi\cap A^2 = \The$. So $\m{B}$ has the $\Q$-\prp{CEP}.
\end{proof}

In particular, we obtain the following generalization of~\cite{Bac72}*{Lemma~1.2}. 

\begin{Corollary}\label{c:rcep iff ep}
Let $\Q$ be any quasivariety. Then $\Q$ has the $\Q$-congruence extension property if and only if $\Q$ has the extension property.
\end{Corollary}

We can now combine these results to obtain the second main result of this section.

\begin{Theorem}\label{t:sepcep}
Let $\Q$ be a relatively congruence-distributive quasivariety such that $\Qrfsi=\Qfsi$ is closed under subalgebras. The following are equivalent:
\begin{enumerate}[label=(\arabic*)]
\item	$\Q$ has the $\Q$-congruence extension property.
\item	$\Q$ has the extension property.
\item	$\Qrfsi=\Qfsi$ has the $\Q$-congruence extension property. 
\item	$\Qrfsi=\Qfsi$ has the extension property.
\end{enumerate}
\end{Theorem}
\begin{proof}
The equivalence of (1) and (2) is a special case of Corollary~\ref{c:rcep iff ep}, and the equivalence of (1) and (3), and the implications from (3) to (4), and from (4) to (3), follow from Theorem~\ref{t:CEPmain}, Corollary~\ref{c:cepsep}, and Proposition~\ref{p:rcep iff ep}, respectively.
\end{proof}

\begin{Corollary}
Let $\V$ be a congruence-distributive variety such that $\Vfsi$ is closed under subalgebras. The following are equivalent:
\begin{enumerate}[label=(\arabic*)]
\item	$\V$ has the congruence extension property.
\item	$\V$ has the extension property.
\item	$\Vfsi$ has the congruence extension property. 
\item	$\Vfsi$ has the extension property.
\end{enumerate}
\end{Corollary}

\begin{Remark}
Even for a congruence-distributive variety $\V$, it is possible for $\Vfsi$ to have the \prp{EP} but not the \prp{CEP}. For example, let $\V$ be the variety generated by the lattice-ordered monoid $\m{C}_4 = \lan \{-2,-1,1,2\},\min,\max,\cdot,1\ran$ with multiplication table
\[
\begin{array}{|r||r|r|r|r|}
\hline
\cdot & -2 & -1  & 1 & 2\\
\hline\hline
-2 &  -2 & -2 & -2 & -2\\
\hline
-1 &  -2 & -1 & -1 & 2\\
\hline
1  &  -2 & -1 & 1 & 2\\
\hline
2  &  -2 & 2 & 2 & 2\\
\hline
\end{array}
\]
The proper subuniverses of $\m{C}_4$ are $C_1=\{1\}$, $C_2=\{-1,1\}$, $C^\delta_2=\{1,2\}$, $C_3 = \{-1,1,2\}$, $C^\delta_3=\{-2,1,2\}$, and $C^*_3 = \{-2,-1,1\}$, up to isomorphism, and any homomorphic image of $\m{C}_4$ is isomorphic to one of its subalgebras. As shown in~\cite{San2?}, the algebra $\m{C}_4$, and hence $\Vfsi$, does not have the \prp{CEP}: just observe that $\The:=\De_{C^*_3}\cup\{\lan -1,-2\ran,\lan -2,-1\ran\}\in\Con\m{C}^*_3$, but $\cg{\m{C}_4}(\The)= C_4\times C_4$. On the other hand, using the fact that $\m{C}^*_3\not\in\Vfsi$, it is easy to confirm that $\Vfsi$ has the \prp{EP}. 
\end{Remark}


\section{The Amalgamation Property} \label{s:ap}

We first recall a useful necessary and sufficient condition for the existence of amalgams. 

\begin{Lemma}[\cite{GL71}*{Lemma~2}] \label{l:gratzer}
Let  $\lan\m{A},\m{B},\m{C},\f_B,\f_C \ran$ be a doubly injective span in a class of similar algebras $\K$ and suppose that 
\begin{enumerate}[label=(\alph*)]
\item for any distinct $x, y\in B$, there exist a $\m{D}_B^{xy}\in \K$ and homomorphisms $\ps_B^{xy}\colon\m{B}\to\m{D}_B^{xy}$ and $\ps_C^{xy}\colon\m{C}\to\m{D}_B^{xy}$ satisfying $\ps_B^{xy} \f_B= \ps_C^{xy} \f_C$ and $\ps_B^{xy}(x)\neq \ps_B^{xy}(y)$;
\item for any distinct $x, y\in C$, there exist a $\m{D}_C^{xy}\in \K$ and homomorphisms $\x_B^{xy}\colon\m{B}\to\m{D}_C^{xy}$ and $\x_C^{xy}\colon\m{C}\to\m{D}_C^{xy}$ satisfying $\x_B^{xy} \f_B= \x_C^{xy} \f_C$ and $\x_C^{xy}(x)\neq \x_C^{xy}(y)$.
\end{enumerate}
Then $\lan\m{A},\m{B},\m{C},\f_B,\f_C \ran$ has an amalgam $\lan\m{D},\ps_B,\ps_C\ran$, where $\m{D}$ is the product of the algebras in the set $\{\m{D}_B^{xy}\mid x,y\in B, x\neq y\}\cup\{\m{D}_C^{xy}\mid x,y\in C,x\neq y\}$.
\end{Lemma}

The following result will play a key role in the proof of our Theorem~\ref{t:APmain}. A slightly weaker version (applying only to varieties) was first proved in~\cite{MMT14} and used to establish a special case of Theorem~\ref{t:APmain} where $\V$ is a variety of semilinear residuated lattices with the \prp{CEP} and $\Vfsi$ is the class of totally ordered members of $\V$. 

\begin{Proposition}[cf.~\cite{MMT14}*{Theorem~9}]\label{p:condition}
Let $\K$ be a subclass of a quasivariety $\Q$ satisfying 
\begin{enumerate}[label=(\roman*)]
\item $\K$ is closed under isomorphisms and subalgebras; 
\item every relatively subdirectly irreducible member of $\Q$ belongs to $\K$;
\item for any $\m{B}\in\Q$ and subalgebra $\m{A}$ of $\m{B}$, if $\The\in\Conq\m{A}$ and  $\m{A}/\The\in \K$, then there exists a $\Phi\in\Conq\m{B}$ such that $\Phi \cap A^2 =\The$ and $\m{B}/\Phi\in \K$;
\item every doubly injective span in $\K$ has an amalgam in $\Q$.
\end{enumerate}
Then $\Q$ has the amalgamation property.
\end{Proposition}
\begin{proof}
Let  $\lan\m{A},\m{B},\m{C},\f_B,\f_C \ran$ be a doubly injective span in $\K$, assuming, by (i), without loss of generality, that $\m{A}$ is a subalgebra of $\m{B}$ and $\m{C}$, and $\f_B$ and $\f_C$ are inclusion maps. We check condition~(a) of  Lemma~\ref{l:gratzer} for the existence of an amalgam, condition~(b) being completely symmetrical. Consider any distinct $x,y\in B$ and let $\Psi$ be a $\Q$-congruence of $\m{B}$ that is maximal with respect to $\lan x,y\ran\notin \Psi$. Then $\m{B}/\Psi$ is a relatively subdirectly irreducible member of $\Q$ and belongs to $\K$, by (ii). Define $\The:=\Psi \cap A^{2}$. The map $\f'_B$ sending $[a]_\The\in\m{A}/\The$ to $[a]_\Psi$ is an embedding of $\m{A}/\The$ into $\m{B}/\Psi$, so $\m{A}/\The\in\K$, by (i), and $\The\in \Conq\m{A}$. Hence, by (iii), there exists a $\Phi\in\Conq\m{C}$ such that $\Phi \cap A^2 = \The$ and $\m{C}/\Phi\in \K$. Moreover, the map $\f'_C$ sending any $[a]_\The\in\m{A}/\The$ to $[a]_\Phi $ is an embedding of $\m{A}/\The$ into $\m{C}/\Phi$.

So $\lan\m{A}/\The,\m{B}/\Psi ,\m{C}/\Phi,\f'_B,\f'_C\ran$ is a  doubly injective span  of  members of $\K$ and, by (iv), has an amalgam $\lan\m{D}_{xy},\x_B,\x_C \ran$ in $\Q$. We define homomorphisms
\[
\ps^{xy}_B\colon\m{B}\to\m{D}_{xy}; \enspace b \mapsto \x_B([b]_\Psi) \quad\text{and}\quad \ps^{xy}_C\colon\m{C}\to\m{D}_{xy}; \enspace c \mapsto \x_C([c]_\Phi).
\]
Then $\ps^{xy}_B(x)\neq \ps^{xy}_B(y)$ (as $\x_B$ is injective and $[x]_\Psi \neq [y]_\Psi $) and for any $a\in A$,
\begin{align*}
\ps^{xy}_B(\f_B(a))
& =\x_B([a]_\Psi)\\
& =\x_B(\f'_B([a]_\The))\\
& =\x_C(\f'_C([a]_\The))\\
& =\x_C([a]_\Phi)\\
& =\ps^{xy}_C(\f_C(a)). \qedhere
\end{align*}
\end{proof}

We now prove the main result of this section.

\begin{Theorem}\label{t:APmain}
Let $\Q$ be any quasivariety with the $\Q$-congruence extension property such that $\Qrfsi$ is closed under subalgebras. The following are equivalent:
\begin{enumerate}[label=(\arabic*)]
\item	$\Q$ has the amalgamation property.
\item	$\Q$ has the one-sided amalgamation property.
\item	$\Qrfsi$ has the one-sided amalgamation property.
\item	Every doubly injective span in $\Qrfsi$ has an amalgam in $\Qrfsi\times\Qrfsi$.
\item Every doubly injective span in $\Qrfsi$ has an amalgam in $\Q$.
\end{enumerate}
\end{Theorem}
\begin{proof}
(1)\,$\Rightarrow$\,(2). Immediate.

(2)\,$\Rightarrow$\,(3). Suppose that $\Q$ has the \prp{1AP} and let  $\lan\m{A},\m{B},\m{C},\f_B,\f_C \ran$ be a doubly injective span in $\Qrfsi$. By assumption, there exist a $\m{D}'\in\Q$,  a homomorphism $\ps'_B\colon\m{B}\to\m{D}'$, and an embedding $\ps'_C\colon\m{C}\to\m{D}'$ such that $\ps'_B \f_B=\ps'_C\f_C$. We may assume without loss of generality that $\m{C}$ is a subalgebra of $\m{D}'$. By Corollary~\ref{c:key}, there exist a $\m{D}\in\Qrfsi$ and a surjective homomorphism $\x\colon\m{D}'\to\m{D}$ such that $\ker(\x)\cap C^2 = \De_C$. Hence $\ps_B:=\x\ps'_B$ is a homomorphism from $\m{B}$ to $\m{D}$, and $\ps_C:=\x\ps'_C$ is an embedding of $\m{C}$ into $\m{D}$ satisfying $\ps_B \f_B=\x\ps'_B\f_B=\x\ps'_C\f_C=\ps_C\f_C$.

(3)\,$\Rightarrow$\,(4).  Suppose that $\Qrfsi$ has the \prp{1AP} and let $\lan\m{A},\m{B},\m{C},\f_B,\f_C \ran$ be any doubly injective span in $\Qrfsi$. By assumption, there exist a $\m{D}_C\in\Qrfsi$,  a homomorphism $\ps^C_B\colon\m{B}\to\m{D}_C$, and an embedding $\ps^C_C\colon\m{C}\to\m{D}_C$ such that $\ps^C_B \f_B=\ps^C_C\f_C$. However, $\lan\m{A},\m{C},\m{B},\f_C,\f_B \ran$ is also a doubly injective span in $\Qrfsi$, so there exist a $\m{D}_B\in\Qrfsi$,  a homomorphism $\ps^B_C\colon\m{C}\to\m{D}_B$, and an embedding $\ps^B_B\colon\m{B}\to\m{D}_B$ such that $\ps^B_C \f_C=\ps^B_B\f_B$. Hence $\lan\m{A},\m{B},\m{C},\f_B,\f_C \ran$ has an amalgam $\lan \m{D},\ps_B,\ps_C\ran$, where $\m{D} =\m{D}_B\times\m{D}_C\in\Qrfsi\times\Qrfsi$, $\ps_B$ maps $x\in B$ to $\lan\ps^B_B(x),\ps^C_B(x)\ran$, and $\ps_C$ maps $x\in C$ to $\lan\ps^B_C(x),\ps^C_C(x)\ran$.

(4)\,$\Rightarrow$\,(5). Immediate.

(5)\,$\Rightarrow$\,(1).  Suppose that every doubly injective span in $\Qrfsi$ has an amalgam in  $\Q$. Since $\Qrsi\subseteq\Qrfsi$ and, by assumption, $\Qrfsi$ is closed under subalgebras, it suffices to observe that condition~(iii) of Proposition~\ref{p:condition} is satisfied for $\K=\Qrfsi$ by Lemma~\ref{l:key}.
\end{proof}

Since a variety $\V$ has the amalgamation property if and only if every doubly injective span of finitely generated algebras in $\V$ has an amalgam in $\V$~\cite{Gra75}, small adjustments to the proofs in this section yield the following variant of Theorem~\ref{t:APmain} for varieties.

\begin{Corollary}\label{c:APvar}
Let $\V$ be any variety with the congruence extension property such that $\Vfsi$ is closed under subalgebras. The following are equivalent:
\begin{enumerate}[label=(\arabic*)]
\item	$\V$ has the amalgamation property.
\item	$\V$ has the one-sided amalgamation property.
\item	$\Vfsi$ has the one-sided amalgamation property.
\item	Every doubly injective span of finitely generated algebras in $\Vfsi$ has an amalgam in $\Vfsi\times\Vfsi$.
\item Every doubly injective span of finitely generated algebras in $\Vfsi$ has an amalgam in $\V$.
\end{enumerate}
\end{Corollary}

\begin{Remark}
The \prp{1AP} cannot be replaced by the \prp{AP} in condition (3) of Theorem~\ref{t:APmain} or Corollary~\ref{c:APvar}. For example, the variety $\cls{DL}$ of distributive lattices is congruence-distributive and has the \prp{CEP} and \prp{AP}, but $\cls{DL}_{_\text{FSI}}$, which up to isomorphism contains only the trivial lattice and two-element lattice, does not have the \prp{AP}. Just observe that any amalgam of a doubly injective span embedding the trivial lattice into the two-element lattice in two different ways must have at least three elements and hence cannot belong to $\cls{DL}_{_\text{FSI}}$.
\end{Remark}

The following result is useful for the study of joins of varieties with the \prp{AP} (see, e.g., the proof of Theorem~\ref{t:bltheorem}). Recall that a subalgebra $\m{A}$ of an algebra $\m{B}$ is a {\em retract} of $\m{B}$ if there exists a homomorphism $\ps\colon\m{B}\to\m{A}$ such that $\ps$ is the identity on $A$. 

\begin{Proposition}\label{p:joinsap}
Let $\V_{1}$ and $\V_{2}$ be varieties of the same signature such that $\V_1\jn\V_2$ is congruence-distributive, and suppose that $\V_{1}$ and $\V_{2}$ have the \prp{AP} and \prp{CEP}, $(\V_1)_{_\text{FSI}}$ and $(\V_2)_{_\text{FSI}}$ are closed under subalgebras, and whenever $\m{A}\in(\V_1)_{_\text{FSI}}\cap(\V_2)_{_\text{FSI}}$ is a subalgebra of $\m{B}\in(\V_1)_{_\text{FSI}}\cup(\V_2)_{_\text{FSI}}$, it is a retract of $\m{B}$.  Then $\V_1\jn\V_2$ has the amalgamation property.
\end{Proposition}
\begin{proof}
Note first that, since $\V_1\cup\V_2$ is a positive universal class, J{\'o}nsson's Lemma~\cite{Jon67} yields $(\V_1\jn\V_2)_{_\text{SI}}\subseteq\V_1\cup\V_2$ and hence $(\V_1\jn\V_2)_{_\text{SI}} =  (\V_1)_{_\text{SI}}\cup(\V_2)_{_\text{SI}}$. However, by the Relativized J{\'o}nsson Lemma~\cite{CD90}*{Lemma~1.5}, every finitely subdirectly irreducible member of a variety embeds into an ultraproduct of its subdirectly irreducible members. So  also $(\V_1\jn\V_2)_{_\text{FSI}}\subseteq\V_1\cup\V_2$, and $(\V_1\jn\V_2)_{_\text{FSI}} =  (\V_1)_{_\text{FSI}}\cup(\V_2)_{_\text{FSI}}$. Now consider any doubly-injective span $\lan\m{A},\m{B},\m{C},\f_B,\f_C \ran$ in $(\V_1\jn\V_2)_{_\text{FSI}}$. Since $\V_1$ and $\V_2$ have the \prp{AP}, we may assume that $\m{A}\in(\V_1)_{_\text{FSI}}\cap(\V_2)_{_\text{FSI}}$ is a subalgebra of $\m{B},\m{C}\in(\V_1)_{_\text{FSI}}\cup(\V_2)_{_\text{FSI}}$. By assumption, there exists a homomorphism $\ps_B\colon\m{B}\to\m{C}$ (since $\m{A}$ is a subalgebra of $\m{C}$) such that $\ps_B$ is the identity on $A$. Let $\ps_C$ be the identity map on $\m{C}$. Clearly, $\ps_B\f_B=\ps_C\f_C$. Hence $(\V_1\jn\V_2)_{_\text{FSI}}$ has the \prp{1AP} and so  $\V_1\jn\V_2$  has the \prp{AP}, by Theorem~\ref{t:APmain}.
\end{proof}

Let us conclude this section by remarking that by considering  completely meet-irreducible $\Q$-congruences in the proof of Theorem~\ref{t:APmain}, we obtain the same result with $\Qrfsi$ replaced by the class $\Qrsi^+$ of trivial or $\Q$-subdirectly irreducible members of $\Q$. In particular, we obtain~\cite{GL71}*{Theorem~3}, which states that a variety $\V$ with the \prp{CEP} such that the class $\Vsi^+$ of trivial or subdirectly irreducible members of $\V$ is closed under subalgebras has the \prp{AP} if and only if every doubly injective span in $\Vsi^+$ has an amalgam in $\V$. Note, however, that while the property that $\Qrfsi$ is closed under subalgebras follows from the existence of equationally definable relative principal congruence meets (satisfied by many quasivarieties serving as algebraic semantics for non-classical logics), a similarly general condition guaranteeing closure under subalgebras is not known for $\Qrsi^+$.


\section{Transferable Injections and Strong Amalgamation}\label{s:further}

We first establish a generalization for classes of similar algebras closed under subalgebras of the well-known fact that a variety has the \prp{TIP} if and only if it has the \prp{CEP} and \prp{AP}~\cite{Bac72}*{Lemma~1.7}. 

\begin{Proposition}\label{p:tipapcep}
Let $\K$ be a class of similar algebras that is closed under subalgebras. Then $\K$ has the transferable injections property if and only if it has the one-sided amalgamation property and extension property.
\end{Proposition}
\begin{proof}
The left-to-right direction is immediate. For the converse, suppose that $\K$ has the \prp{1AP} and \prp{EP} and consider any injective span $\lan\m{A},\m{B},\m{C},\f_B,\f_C \ran$ in $\K$. Let $\m{C}':=\f_C[\m{A}]$. By assumption, $\m{C}'\in\K$. Moreover, $\f_C\colon\m{A}\to\m{C}'$ is surjective, so, since $\K$ has the \prp{EP}, there exist a $\m{D}'\in \K$, a homomorphism $\ps'_B\colon\m{B}\to\m{D}'$, and an embedding $\ps'_C\colon\m{C}'\to\m{D}'$ such that $\ps'_B \f_B=\ps'_C\f_C$. Let $\nu\colon\m{C}'\to\m{C}$ be the inclusion map. Since $\K$ has the \prp{1AP}, there exist for the doubly injective span $\lan\m{C}',\m{D}',\m{C},\ps'_C,\nu\ran$, a $\m{D}\in \K$, a homomorphism $\x\colon\m{D}'\to\m{D}$, and an embedding $\ps_C\colon\m{C}\to\m{D}$ such that $\x\ps'_C=\ps_C\nu$. Let $\ps_B:=\x\ps'_B$. Then $\ps_B \f_B = \x\ps'_B\f_B=\x\ps'_C\f_C=\ps_C\nu\f_C=\ps_C\f_C$. Hence $\K$ has the \prp{TIP}.
\end{proof}

Combining now Proposition~\ref{p:tipapcep} with our earlier results for the \prp{1AP} and \prp{EP}, we are able to transfer results for the \prp{TIP} from a quasivariety to the class of its relatively finitely subdirectly irreducible members, and back again.

\begin{Lemma}\label{l:tipfsi}
If a quasivariety $\Q$ has the transferable injections property, then $\Qrfsi$ has the transferable injections property.
\end{Lemma}
\begin{proof}
Suppose that $\Q$ has the \prp{TIP} and let $\lan\m{A},\m{B},\m{C},\f_B,\f_C \ran$ be an injective span in $\Qrfsi$. Then there exist a $\m{D}\in \Q$, a homomorphism $\ps_B\colon\m{B}\to\m{D}$, and an embedding $\ps_C\colon\m{C}\to\m{D}$ such that $\ps_B \f_B=\ps_C\f_C$. Without loss of generality, we may assume that $\m{C}$ is a subalgebra of $\m{D}$ and $\ps_C$ is the inclusion map. By Corollary~\ref{c:key}, there exist a $\m{D^*}\in\Qrfsi$ and a surjective homomorphism $\x\colon\m{D}\to\m{D^*}$ such that $\ker(\x)\cap C^2 =\De_C$. Let $\ps^*_B:=\x\ps_B$ and $\ps^*_C:=\x\ps_C$. Then $\ker(\ps^*_C)=\ker(\x)\cap C^2 =\De_C$, so $\ps^*_C$ is an embedding, and $\ps^*_B \f_B=\x\ps_B\f_B=\x\ps_C\f_C=\ps^*_C\f_C$. Hence  $\Qrfsi$ has the \prp{TIP}.
\end{proof}

\begin{Theorem}\label{t:mainTIP}
Let $\Q$ be a relatively congruence-distributive quasivariety such that $\Qrfsi=\Qfsi$ is closed under subalgebras.  Then $\Q$ has the transferable injections property if and only if $\Qrfsi=\Qfsi$ has the transferable injections property.
\end{Theorem}
\begin{proof}
The left-to-right direction follows directly from Lemma~\ref{l:tipfsi}. For the converse, suppose that $\Qrfsi$ has the \prp{TIP}. Then $\Qrfsi$ has the \prp{1AP} and the \prp{EP}, by Proposition~\ref{p:tipapcep}, and hence $\Q$ has the $\Q$-\prp{CEP}, by Theorem~\ref{t:sepcep}. Moreover, $\Q$ has the \prp{AP}, by Theorem~\ref{t:APmain}. So $\Q$ has the \prp{TIP}, by Proposition~\ref{p:tipapcep}.
\end{proof}

We now recall a useful characterization of the \prp{SAP} and a transfer theorem for \prp{SE}.

\begin{Theorem}[\cite{Isb1966}]\label{t:isbell}
Let $\Q$ be a quasivariety. Then $\Q$ has the strong amalgamation property if and only if $\Q$ has the amalgamation property and surjective epimorphisms.
\end{Theorem}

\begin{Theorem}[\cite{Cam18}*{Theorem~22}]\label{t:campercholi}
Let $\V$ be an arithmetical variety such that $\Vfsi$ is a universal class. Then $\V$ has surjective epimorphisms if and only if $\Vfsi$ has surjective epimorphisms.
\end{Theorem}

Combining Theorems~\ref{t:isbell} and~\ref{t:campercholi} with Corollary~\ref{c:APvar} yields the following result.

\begin{Corollary}\label{c:mainsa}
Let $\V$ be an arithmetical variety with the congruence extension property such that $\Vfsi$ is a universal class. Then $\V$ has the strong amalgamation property if and only if $\Vfsi$ has the one-sided amalgamation property and surjective epimorphisms.
\end{Corollary}

These results do not provide a transfer theorem for the \prp{SAP}, however. To obtain such a theorem, at least in one direction, we make use of a well-known characterization of the \prp{SE} property. Let $\K$ be a class of similar algebras and consider any $\m{A}\in\K$. We say that $\m{A}$ is a \emph{$\K$-epic subalgebra} of $\m{B}\in\K$ if $\m{A}$ is a subalgebra of $\m{B}$ and for every $\m{C}\in\K$ and all homomorphisms $\ps_1,\ps_2\colon \m{B}\to\m{C}$, if $\ps_1$ and $\ps_2$ are equal on their restriction to $A$, then $\ps_1=\ps_2$. It is easy to see that $\m{A}$ is a $\K$-epic subalgebra of $\m{B}$ if and only if $\m{A}$ is a subalgebra of $\m{B}$ and the inclusion homomorphism of $\m{A}$ into $\m{B}$ is an epimorphism in $\K$. The following result is folklore (see, e.g.,~\cite{MorWan2020}*{Lemma~3.1}), but we include a proof for completeness.

\begin{Lemma}\label{l.epic}
Let $\K$ be a class of similar algebras closed under subalgebras. Then $\K$ has surjective epimorphisms if and only if no member of $\K$ has a proper $\K$-epic subalgebra.
\end{Lemma}

\begin{proof}
Let $\m{A},\m{B}\in\K$ and suppose that $\f\colon\m{A}\to\m{B}$ is a non-surjective $\K$-epimorphism. Then $\f[\m{A}]$ is a proper subalgebra of $\m{B}$. Observe that if $\m{C}\in\K$ and $\ps_1,\ps_2\colon\m{B}\to\m{C}$ are homomorphisms that coincide on their restriction to $\f[A]$, then $\ps_1\f=\ps_2\f$ and hence $\ps_1=\ps_2$ since $\f$ is an epimorphism. It follows that $\f[\m{A}]$ is a proper $\K$-epic subalgebra of $\m{B}$.

For the converse, let $\m{B}\in\K$ and suppose that $\m{A}$ is a proper $\K$-epic subalgebra of $\m{B}$. Then $\m{A}\in\K$ since $\K$ is closed under subalgebras. Hence the inclusion map of $\m{A}$ into $\m{B}$ is a non-surjective $\K$-epimorphism between members of $\K$.
\end{proof}

The following theorem strengthens a result from~\cite{FusGal2022}.

\begin{Theorem}\label{t:mainsa}
Let $\V$ be an arithmetical variety with the congruence extension property such that $\Vfsi$ is a universal class. If every doubly injective span in $\Vfsi$ has a strong amalgam in $\V$, then $\V$ has the strong amalgamation property.
\end{Theorem}
\begin{proof}
Suppose that every doubly injective span in $\Vfsi$ has a strong amalgam in $\V$. Then $\V$ has the \prp{AP}, by Corollary~\ref{c:APvar}. Hence, by Theorems~\ref{t:isbell} and~\ref{t:campercholi}, it is enough to show that $\Vfsi$ has \prp{SE}, or, equivalently, by Lemma~\ref{l.epic}, that no member of $\Vfsi$ has a proper $\Vfsi$-epic subalgebra. Let $\m{B}\in\Vfsi$, let $\m{A}$ be a proper subalgebra of $\m{B}$, and let $\iota\colon\m{A}\to\m{B}$ be the inclusion map. By assumption, the span $(\m{A},\m{B},\m{B},\iota,\iota)$ has a strong amalgam $(\m{D},\ps_1,\ps_2)$ in $\V$. Let $\x\colon \m{D}\to\prod_{i\in I} \m{D}_i$ be a subdirect representation of $\m{D}$, so that in particular $\m{D}_i\in\Vfsi$ for each $i\in I$. Since $\m{D}$ is a strong amalgam, there exists an $i\in I$ such that $\x\ps_1(i)\neq\x\ps_2(i)$. On the other hand, since $\m{D}$ is an amalgam, it follows that $\x\ps_1\iota = \x\ps_2\iota$. This implies that $\iota$ is not a $\K$-epimorphism, proving the theorem.
\end{proof}


\section{Decidability}\label{s:decidability}

To fix terminology, let us call a variety $\V$ {\em finitely generated} if it is generated as a variety by a given finite set of finite algebras of finite signature, and {\em residually small} if there exists a bound on the size of its subdirectly irreducible members. It is known that if a residually small congruence-distributive variety $\V$ has the \prp{AP}, then it has the \prp{CEP}~\cite{Kea89}*{Corollary~2.11}.

Consider any finitely generated congruence-distributive variety $\V$ such that $\Vfsi$ is closed under subalgebras. By J{\'o}nsson's Lemma~\cite{Jon67}, there exists and can be constructed a finite set $\Vfsi^*\subseteq\Vfsi$ of finite algebras such that each $\m{A}\in\Vfsi$ is isomorphic to some $\m{A}^*\in\Vfsi^*$. Hence, by Corollary~\ref{c:CEPvar}, it can be decided if $\V$ has the \prp{CEP} by checking if each member of $\Vfsi^*$ has the \prp{CEP}. Since $\V$ is clearly residually small, if $\V$ does not have the \prp{CEP}, it cannot have the \prp{AP}. Otherwise, $\V$ has the \prp{CEP} and, by Corollary~\ref{c:APvar}, it can be decided if $\V$ has the \prp{AP} by checking --- by considering the finitely many finite algebras in $\Vfsi^*$ --- if $\Vfsi$ has the \prp{1AP}. Finally, if $\V$ is also arithmetical, to check if $\V$ has \prp{SE}, it suffices to check --- again, considering the algebras in $\Vfsi^*$ --- if $\Vfsi$ has  \prp{SE}~\cite{Cam18}*{Theorem~22}, and $\V$ then has the \prp{SAP} if and only if it has \prp{SE} and the \prp{AP}~\cite{Isb1966}. Hence we have established the following result.

\begin{Theorem}\label{t:decidability}
Let $\V$ be a finitely generated congruence-distributive variety such that $\Vfsi$ is closed under subalgebras. There exist effective algorithms to decide if $\V$ has the congruence extension property, amalgamation property, or transferable injections property. If $\V$ is also arithmetical, then there exist effective algorithms to decide if $\V$ has surjective epimorphisms or the strong amalgamation property.
\end{Theorem}


\section{A Case Study: Varieties of BL-Algebras}\label{s:BLalgebras}

BL-algebras, introduced by H{\'a}jek in~\cite{haj98} as an algebraic semantics for his basic fuzzy logic of continuous t-norms, have been studied intensively over the past twenty five years, largely in the framework of substructural logics and residuated lattices (see, e.g.,~\cites{AgMo03,Mon06,JM10,MMT14,AB21,FZ21,Fus22}). In this section, we use the tools developed in previous sections to contribute to the development of a (still incomplete) description of the varieties of BL-algebras that have the~\prp{AP}.

A \emph{BL-algebra} is an algebra $\m A=\lan A,\mt,\jn,\cdot,\to,0,1\ran$ satisfying
\begin{enumerate}[label=(\roman*)]
\item	$\lan A,\mt,\jn,0,1 \ran$ is a bounded lattice with order $a\le b :\Longleftrightarrow a\mt b=a$;
\item	$\lan A,\cdot,1 \ran$ is a commutative monoid;
\item $\to$ is the residual of $\cdot$, i.e., $a \cdot b \le c\iff b \le a \to c$\, for all $a,b,c\in A$;
\item $a \mt b = a \cdot (a\to b)$\, and\, $(a\to b)\jn (b\to a)=1$\, for all $a,b\in A$.
\end{enumerate}
The class of BL-algebras forms a congruence-distributive variety $\cls{BL}$ with the \prp{CEP}, and $\cls{BL}_{_\text{FSI}}$ is the positive universal class consisting of all totally ordered BL-algebras (i.e., such that $\lan A,\le\ran$ in the preceding definition is a chain). BL-algebras hence form a subvariety of the variety of semilinear integral bounded residuated lattices (also known as MTL-algebras)~\cites{est:mtl,FU2019}.

Notable subvarieties of $\cls{BL}$ include (i) the variety $\cls{MV}$ of {\em MV-algebras} consisting of BL-algebras satisfying $a\jn b = (a\to b)\to b$ for all $a,b\in A$; (ii) the variety $\cls{G}$ of {\em G{\"o}del algebras} consisting of BL-algebras satisfying $a\mt b= a\cdot b$ for all $a,b\in A$; (iii) the variety $\cls{P}$ of {\em product algebras} consisting of BL-algebras satisfying $a\mt(a\to 0)=0$ and $(a\to 0)\jn ((a\to(a\cdot b))\to b)=1$ for all $a,b\in A$. The varieties $\cls{MV}$, $\cls{G}$, and $\cls{P}$ are generated by algebras $\mathbfL$, $\m{G}$, and $\m{P}$ of the form $\lan [0,1],\min,\max,\star,\to_\star,0,1\ran$, where $x\star y$ is $\max(0,x+y-1)$, $\min(x,y)$, and $xy$, respectively. Every member of $\cls{BL}_{_\text{FSI}}$ (i.e., every totally ordered BL-algebra) can be constructed as a certain ordinal sum of members of $\cls{MV}_{_\text{FSI}}$ and $\cls{P}_{_\text{FSI}}$~\cite{AgMo03}*{Theorem~3.7}. 

Let us briefly recall some further relevant facts about $\cls{MV}$, $\cls{G}$, and $\cls{P}$. First, given any totally ordered Abelian group $\m{L}=\lan L,+,-,0,\le\ran$ and $u\in L$ with $u\ge 0$, defining $a\cdot b:= \max(a+b-u,0)$ and $a\to b:=\min(u-a+b,u)$ yields a totally ordered MV-algebra $\Ga(\m{L},u):=\lan [0,u],\min,\max,\cdot,\to,0,u\ran$. Each proper non-trivial subvariety of $\cls{MV}$ is generated by a non-empty finite set of algebras of the form $\m{S}_n:=\Ga(\m{Z},n)$ or $\m{S}_n^\omega:=\Ga(\m{Z}\times_{_\text{lex}}\m{Z},\lan n,0\ran)$, where $\m{Z}$ is the ordered group of integers, $\m{Z}\times_{_\text{lex}}\m{Z}$ is the lexicographic product of two copies of $\m{Z}$, and $n\in\N^{>0}$~\cite{Kom81}*{Theorem 4.11}. Notably, $\m{S}_1$ generates the variety $\cls{BA}$ of Boolean algebras, and  $\m{S}^\omega_1$, known as the Chang algebra, generates a variety denoted by $\cls{C}$.  Each proper non-trivial subvariety $\cls{G}_n$ of $\cls{G}$ is generated by the algebra $\m{G}_n:=\lan \{0,\frac{1}{n},\dots,\frac{n-1}{n},1\},\min,\max,\min,\to,0,1\ran$, where $n\in\N^{>0}$ and $a\to b=1$ if $a\le b$, otherwise $a\to b=b$. Finally, the only proper non-trivial subvariety of $\cls{P}$ is $\cls{BA}$ (which coincides with $\cls{G}_2$).

The non-trivial subvarieties of $\cls{MV}$, $\cls{G}$, and $\cls{P}$ that have the \prp{AP} are precisely the varieties generated by $\m{S}_n$ ($n\in\N^{>0}$), $\m{S}^\omega_n$ ($n\in\N^{>0}$), $\cls{MV}$, $\cls{G}$, $\cls{G}_3$, and $\cls{P}$~\cite{DiNLe00}*{Theorem~13}. Note also that $\cls{BL}$ has the \prp{AP}~\cite{Mon06}*{Theorem~3.7}, and~\cite{AB21}*{Theorem~6} gives a complete description of the varieties of BL-algebras having the~\prp{AP} that are generated by a totally ordered BL-algebra built as an ordinal sum of finitely many members of $\cls{MV}_{_\text{FSI}}$ and $\cls{P}_{_\text{FSI}}$. Here, we consider subvarieties of  $\cls{BL}_1 := \cls{MV}\jn\cls{G}\jn\cls{P}$, each of which is generated by a class of ``one-component'' totally ordered BL-algebras, i.e., members of $\cls{MV}_{_\text{FSI}}$,  $\cls{G}_{_\text{FSI}}$, and  $\cls{P}_{_\text{FSI}}$. 

\begin{Lemma}\label{l:bl1fail}
Let $\V$ be a subvariety of $\cls{BL}_1$ satisfying $\V\not\subseteq\cls{MV}$ and $\m{S}_n\in\V$ for some $n\in\N^{>1}$. Then $\V$ does not have the amalgamation property.
\end{Lemma}
\begin{proof}
Consider a doubly injective span $\lan\m{A},\m{B},\m{C},\f_B,\f_C\ran$ in $\Vfsi$, where $\m{A}$ is the two-element Boolean algebra, $\m{B}$ is $\m{S}_n\in\V$ for some $n\in\N^{>1}$, and $\m{C}\not\in\cls{MV}$. Suppose that there exist a $\m{D}\in\Vfsi$ and an embedding $\ps_C$ of $\m{C}$ into $\m{D}$. Then also $\m{D}\not\in\cls{MV}$, so $\m{D}\in\cls{G}_{_\text{FSI}}\cup\cls{P}_{_\text{FSI}}$. Since $\m{B}$ is a finite totally ordered MV-algebra, it is simple. Hence any homomorphism $\ps_B\colon\m{B}\to\m{D}$ is either trivial (i.e., maps all elements of $\m{B}$ to one element in $\m{D}$), which is not possible, since $\m{D}$ is non-trivial, or injective, which is not possible, since $\m{B}$ does not embed into any totally ordered G{\"o}del algebra or product algebra. So $\Vfsi$ does not have the \prp{1AP} and it follows, by Theorem~\ref{t:APmain}, that $\V$ does not have the \prp{AP}.
\end{proof}

\begin{Theorem}\label{t:bltheorem}
In addition to the varieties of MV-algebras generated by one totally ordered MV-algebra, there are precisely ten non-trivial subvarieties of $\cls{BL}_1 $ that have the amalgamation property:  $\cls{G}$, $\cls{G}_3$, $\cls{P}$, $\cls{G}\jn\cls{P}$, $\cls{G}\jn\cls{C}$,  $\cls{G}_3\jn\cls{P}$,  $\cls{G}_3\jn\cls{C}$, $\cls{P}\jn\cls{C}$, $\cls{G}_3\jn\cls{P}\jn\cls{C}$, and $\cls{G}\jn\cls{P}\jn\cls{C}$.
\end{Theorem}
\begin{proof}
Let $\V$ be any non-trivial subvariety of $\cls{BL}_1$. Then $\Vfsi$ consists of members of $\cls{G}_{_\text{FSI}}$, $\cls{MV}_{_\text{FSI}}$, and $\cls{P}_{_\text{FSI}}$. The cases where $\Vfsi$ is included in one of these classes are clear from the previous remarks. Moreover, it follows from~\cite{AB21}*{Theorem~3.3} that if $\V$ contains $\m{G}_n$  for some $n>3$ but not $\m{G}$, then  $\V$ does not have the \prp{AP}. Suppose now that $\V\not\subseteq\cls{MV}$ and $\V\cap\cls{MV}$ is non-trivial. Then $\V\cap\cls{MV}$ is either $\cls{MV}$ or generated by a non-empty finite set of algebras of the form $\m{S}_n$ or $\m{S}_n^\omega$ ($n\in\N^{>0}$). If $\V\cap\cls{MV}\not\in\{\cls{G}_2,\cls{C}\}$, then, since $\m{S}_n$ is a subalgebra of $\m{S}_n^\omega$ for each $n\in\mathbb{N}$, it follows that $\m{S}_n\in\V$ for some $n\in\N^{>1}$, and hence, by Lemma~\ref{l:bl1fail}, that $\V$ does not have the \prp{AP}. 

It remains therefore to show that each of $\cls{G}\jn\cls{P}$, $\cls{G}\jn\cls{C}$,  $\cls{G}_3\jn\cls{P}$,  $\cls{G}_3\jn\cls{C}$, $\cls{P}\jn\cls{C}$, $\cls{G}_3\jn\cls{P}\jn\cls{C}$, and $\cls{G}\jn\cls{P}\jn\cls{C}$ has the \prp{AP}. Clearly, the two-element Boolean algebra $\m{G}_2$ is the only non-trivial algebra common to the finitely subdirectly irreducible members of the varieties in these joins. Hence, by Proposition~\ref{p:joinsap}, it suffices to observe that $\m{G}_2$ is a retract of any non-trivial member of $\cls{G}_{_\text{FSI}}$, $\cls{P}_{_\text{FSI}}$, and $\cls{C}_{_\text{FSI}}$. This follows directly from a general result of~\cite{CT02}*{Theorem~4.5} describing BL-algebras with a Boolean retract, but we may also define a suitable retraction explicitly.  Given any non-trivial $\m{A}\in\cls{G}_{_\text{FSI}}\cup\cls{P}_{_\text{FSI}}\cup\cls{C}_{_\text{FSI}}$, define $\f\colon A\to\{0,1\}$ by mapping $a\in A$ to $0$ if and only if $a^n=0$ for some $n\in\N$, where $a^0:=1$ and $a^{k+1}:=a^k\cdot a$ ($k\in\N$). It is straightforward to verify that $\f$ is a retraction from $\m{A}$ onto $\m{G}_2$.
\end{proof}

\section*{Acknowledgement}

This research was supported by Swiss National Science Foundation grant 200021$\_$165850.


\bibliographystyle{elsarticle-num-names}
\bibliography{bibliography} 

\end{document}